\documentclass[oneside,english]{amsart}
\usepackage[T1]{fontenc}
\usepackage[latin9]{inputenc}
\usepackage{amsthm}
\usepackage{amstext}
\usepackage{amssymb}
\usepackage{esint}
\usepackage{amscd}
\usepackage{color}
\makeatletter
\numberwithin{equation}{section}
\numberwithin{figure}{section}
\theoremstyle{plain}
\newtheorem{thm}{\protect\theoremname}[section]

\theoremstyle{plain}
\theoremstyle{definition}

\theoremstyle{plain}
\newtheorem{lem}[thm]{\protect\lemmaname}
\newtheorem{cor}[thm]{\protect\corollaryname}
\theoremstyle{plain}

\theoremstyle{plain}
\makeatother

\usepackage{babel}
\providecommand{\definitionname}{Definition}
\providecommand{\lemmaname}{Lemma}
\providecommand{\theoremname}{Theorem}
\providecommand{\corollaryname}{Corollary}
\providecommand{\remarkname}{Remark}
\providecommand{\propositionname}{Proposition}
	
\DeclareMathOperator{\loc}{loc}

\DeclareMathOperator{\ess}{ess}

\DeclareMathOperator{\R}{R}

\begin{document}

\title[Compactness of weighted Sobolev trace operators]{Compactness of weighted Sobolev trace operators and non-linear Steklov problems}

\author{Alexander Menovschikov$^{*}$ and Alexander Ukhlov$^{**}$}
\thanks{$^{*}$The first author was supported by the RSF grant No.~25-71-00064}
\thanks{$^{**}$Corresponding author: ukhlov@math.bgu.ac.il}

\begin{abstract}
We prove the compactness of weighted Sobolev trace operators in outward cuspidal domains by using composition operators on Sobolev spaces. This result allows us to formulate the non-linear Steklov problem in outward cuspidal domains in a correct functional setting and to establish the existence of its non-trivial solution.
\end{abstract}
\maketitle


\section{Introduction}

In recent years, the geometric analysis of partial differential equations in cuspidal domains has attracted considerable attention; see, for instance, \cite{AM20,B22,GPU24,GU17,KLP,KUZ,MP,MU26,NT}.
Remark that a complicated boundary value problem in non-Lipschitz domains is the Steklov problem, whose analysis is strongly affected by the lack of compactness of Sobolev trace operators on singular boundaries \cite{NT}. In \cite{GV} the approach to compactness of trace operators based on weights associated with the boundary geometry was suggested. On this base, the recent work \cite{GGU25} established the solvability of the reduced weighted Steklov $p$-eigenvalue problem in outward cuspidal domains. The reduced problem is understood as the formulation with an orthogonality condition, excluding mean oscillations on the boundary.

In the present work we consider the {\it non-reduced} weighted Steklov $(p,q)$-eigenvalue problem in outward cuspidal domains. The boundedness of Sobolev trace operators in outward cuspidal domains $\Omega_{\gamma}\subset\mathbb{R}^n$ can be seen, for example, in \cite{M,MP}. However, the compactness of Sobolev trace operators in outward cuspidal domains $\Omega_{\gamma}\subset\mathbb{R}^n$ is a complicated long-standing problem in the geometric analysis of Sobolev spaces. Therefore, the first part of the article is devoted to compactness of weighted Sobolev trace operators in outward cuspidal domains. The main aim is to describe the class of weights that ensure compactness of the trace operator, depending on the sharpness of the cusp.

Recall the notion of outward $\gamma$-cuspidal domains \cite{GV,KUZ}. Let $g:[0,1] \to [0,\infty)$ be a function defined by $g(t) = t^{\alpha}$, where $\alpha=(\gamma-1)/(n-1)$, $n<\gamma<\infty$. Denote $x'=(x_1,\dots,x_{n-1})$. Then an outward $\gamma$-cuspidal domain $\Omega_{\gamma}\subset\mathbb{R}^n$, $n\geq 2$, is defined by
$$
\Omega_{\gamma}=
\left\{(x',x_n): \sqrt{x_1^2+\dots+x_{n-1}^2}<g(x_n),\ 0<x_n\leq 1\right\}
\cup B^n\left((0,2), \sqrt{2}\right),
$$
where $B^n\left((0,2), \sqrt{2}\right)\subset\mathbb{R}^n$ denotes the open ball of radius $\sqrt{2}$ centered at $(0,2)\in\mathbb{R}^{n-1}\times\mathbb{R}$.

In the present work, by using the composition operators theory \cite{VU04,VU05}, we prove the compactness of the Sobolev trace operator in outward $\gamma$-cuspidal domains
\[
T: W^{1,p}(\Omega_\gamma) \to L^q(\partial\Omega_\gamma, w_\gamma),
\]
for $1<p<n$ and $1<q<\tfrac{p(n-1)}{n-p}$, with the weight
\[
w_{\gamma}(x) = x_n^{\frac{(\gamma - n)(1+p(n-2))}{(n-p)(n-1)}},\qquad x=(x',x_n)\in\partial\Omega_\gamma,
\]
where the weight $w_\gamma$ is defined by the tangential Jacobian of the cusp-straightening mapping and compensates for the distortion of the boundary measure caused by the cusp geometry. So, we refine the trace embedding theorem from \cite{GV}. In addition, we establish the sharpness of this trace embedding within the class of power weights. Consequently, we propose a new approach to trace theorems in cuspidal domains, based on weights generated by composition operators.

Remark that the geometric theory of composition operators on Sobolev spaces goes back to~\cite{U93} and forms a significant part of the modern geometric analysis of Sobolev spaces; see, for example,~\cite{HK12,KKSS14,MU25,PV25}.

By using the compactness of the Sobolev trace operator, we consider the non-reduced weighted Schr\"odinger--Steklov $(p,q)$-eigenvalue problem, for $1<p<\infty$ and $p<q<\tfrac{p(n-1)}{n-p}$,
\begin{equation*}
\begin{cases}
-\mathrm{div}(|\nabla u|^{p-2}\nabla u) + |u|^{p-2}u = 0 & \text{in } \Omega_{\gamma},\\
|\nabla u|^{p-2} \nabla u \cdot \nu = \lambda\, w_{\gamma}\,\|u\|_{L^{q}(\partial \Omega_{\gamma}, w_{\gamma})}^{p-q}\, |u|^{q-2}u & \text{on } \partial\Omega_{\gamma},
\end{cases}
\end{equation*}
in outward $\gamma$-cuspidal domains.

Namely, we prove that if $\Omega_{\gamma}$ is an outward cuspidal domain, then for the weighted Schr\"odinger--Steklov $(p,q)$-eigenvalue problem with $1<p<\infty$ and $p<q<\tfrac{p(n-1)}{n-p}$, there exists a non-trivial solution $u\in W^{1,p}(\Omega_{\gamma})$. Moreover, the first non-trivial eigenvalue $\lambda_{p,q}$ is given by
\begin{multline*}
\lambda_{p,q}=
\inf \left\{\frac{\|\nabla v\|_{L^p(\Omega_{\gamma})}^p+\|v\|_{L^p(\Omega_{\gamma})}^p}{\|v\|_{L^q(\partial \Omega_{\gamma}, w_{\gamma})}^p} : v \in W^{1,p}(\Omega_{\gamma}) \setminus \{0\}\right\}\\
=\frac{\|\nabla u\|_{L^p(\Omega_{\gamma})}^p+\|u\|_{L^p(\Omega_{\gamma})}^p}{\|u\|_{L^q(\partial \Omega_{\gamma}, w_{\gamma})}^p}.
\end{multline*}

Thus, by using the weight $w_\gamma$ defined by the cusp function, we prove the solvability of the weighted Schr\"odinger--Steklov $(p,q)$-eigenvalue problem in outward cuspidal domains. Furthermore, we obtain lower bounds for the first non-trivial Steklov eigenvalue. These results contribute to the spectral theory of nonlinear operators in singular domains.

\section{Sobolev spaces and composition operators}

\subsection{Sobolev spaces}

Let us recall the basic notions of the Sobolev spaces.
Let $\Omega$ be an open subset of $\mathbb R^n$. The Sobolev space $W^{1,p}(\Omega)$, $1<p<\infty$, is defined \cite{M}
as a Banach space of locally integrable weakly differentiable functions
$u:\Omega\to\mathbb{R}$ equipped with the following norm:
$$
\|u\|_{W^{1,p}(\Omega)}=\left(\int_{\Omega}|\nabla u(x)|^p\,dx\right)^\frac{1}{p}+\left(\int_{\Omega}|u(x)|^p\,dx\right)^\frac{1}{p},
$$
where $\nabla u$ is the weak gradient of the function $u$, i.~e. $ \nabla u = (\frac{\partial u}{\partial x_1},...,\frac{\partial u}{\partial x_n})$. In accordance with the non-linear potential theory \cite{MH72}, we consider elements of Sobolev spaces $W^{1,p}(\Omega)$ as equivalence classes up to a set of $p$-capacity zero \cite{M}.

The following result can be found, for example, in \cite[Proposition~9.1]{B11} and \cite{M}.
\begin{lem}
\label{Xuthm}
The space $W^{1,p}(\Omega)$, $1<p<\infty$, is a real separable and reflexive Banach space.
\end{lem}

Let us recall the notion of outward $\gamma$-cuspidal domains in a form convenient for our purposes.
Let $g:[0,1] \to [0,\infty)$ be a function defined by $g(t) = t^{\alpha}$, where $\alpha=(\gamma-1)/(n-1)$, $n<\gamma<\infty$. Then an outward $\gamma$-cuspidal domain $\Omega_{\gamma}\subset\mathbb{R}^n$, $n\geq 2$, is defined by
\begin{equation}
\label{cusp}
\Omega_{\gamma}= \{x = (x_1, x_2, \dots, x_n) \in \mathbb{R}^n: 0 < x_n < 1, 0 < x_i < g(x_n), i = 1, \dots, n-1\}.
\end{equation}

This domain $\Omega_{\gamma}$ is bi-Lipschitz equivalent to the radially symmetric outward $\gamma$-cuspidal domain, defined in Introduction, and therefore supports the same embedding and trace theorems.
 If $\gamma = n$, then $\Omega_n$ is an $n$-dimensional simplex.

Let $E\subset\mathbb R^n$ be a Borel set. Then $E$ is said to be an $H^{m}$-rectifiable set \cite{Fe69}, if $E$ is of Hausdorff dimension $m$, and there exists a countable collection $\{\varphi_i\}_{i\in\mathbb{N}}$ of Lipschitz continuous mappings
$$
\varphi_i: \mathbb R^m\to\mathbb R^n, 
$$   
such that the $m$-Hausdorff measure $H^m$ of the set $E\setminus \bigcup_{i=1}^{\infty}\varphi_i(\mathbb R^m)$
is zero. 

Let $\Omega\subset\mathbb R^n$ be a domain with $H^{n-1}$-rectifiable boundary $\partial\Omega$ and $w:\partial\Omega\to\mathbb R$ be a non-negative continuous function. We consider the weighted Lebesgue space $L^{p}(\partial\Omega, w)$ with the following norm
$$
\|u\|_{L^p(\partial\Omega, w)}=\left(\int_{\partial\Omega}|u(x)|^p w(x)\,d\mathcal{H}^{n-1}(x)\right)^\frac{1}{p},
$$
where $d\mathcal{H}^{n-1}$ is the $(n-1)$-dimensional Hausdorff measure on $\partial\Omega$.

Let us recall the Sobolev embedding theorem in outward $\gamma$-cuspidal domains \cite{GG94,GU09}.

\begin{thm}\label{thmemb}
Let $\Omega_{\gamma}\subset\mathbb R^n$ be an outward $\gamma$-cuspidal domain defined by \eqref{cusp}. Suppose $1<p<\gamma$. Then the embedding operator
$$
W^{1,p}(\Omega_{\gamma})\hookrightarrow L^{q}(\Omega_{\gamma})
$$
is compact for any $1<q<p_{\gamma}^*$, where $p_{\gamma}^*={\gamma p}/{(\gamma -p)}$.
\end{thm}

\subsection{Composition operators on Sobolev spaces} 
The seminormed Dirichlet-Sobolev space $L^{1,p}(\Omega)$ in a domain $\Omega\subset\mathbb R^n$
is the space of all locally integrable weakly differentiable functions with the following seminorm:
\[
\|u\|_{L^{1,p}(\Omega)}=\biggr(\int_{\Omega}|\nabla u(x)|^{p}\, dx\biggr)^{\frac{1}{p}}.
\]

Let $\Omega$ and $\widetilde{\Omega}$ be domains in the Euclidean space $\mathbb R^n$. Then a homeomorphism
$\varphi:\Omega\to \widetilde{\Omega}$ belongs to the Sobolev class $W^{1,p}_{\loc}(\Omega)$,
$1\leq p\leq\infty$, if its coordinate functions $\varphi_j$ belong to $W^{1,p}_{\loc}(\Omega)$, $j=1,\dots,n$.
In this case the formal Jacobi matrix
$D\varphi(x)=\left(\frac{\partial \varphi_i}{\partial x_j}(x)\right)$, $i,j=1,\dots,n$,
and its determinant (Jacobian) $J(x,\varphi)=\det D\varphi(x)$ are well defined at
almost all points $x\in \Omega$. The norm $|D\varphi(x)|$ of the matrix
$D\varphi(x)$ is the norm of the corresponding linear operator $D\varphi (x):\mathbb R^n \rightarrow \mathbb R^n$ defined by the matrix $D\varphi(x)$.

Recall that a Sobolev homeomorphism $\varphi: \Omega \to \widetilde{\Omega}$ of the class $W^{1,1}_{\loc}(\Omega)$ has finite distortion \cite{VGR} if
\[
D\varphi(x)=0\,\,\, \text{a.e. on the set}\,\,\, Z=\{x \in \Omega:|J(x,\varphi)|=0\}.
\]

The results of \cite{U93,VU98} give the analytic description of composition operators on seminormed Sobolev spaces  and assert that
\begin{thm}
\label{CompTh} \cite{U93} Let $\varphi:\Omega\to\widetilde{\Omega}$ be a homeomorphism
between two domains $\Omega$ and $\widetilde{\Omega}$ $\subset\mathbb R^n$. Then $\varphi$ induces by the composition rule $\varphi^{\ast}(f)=f\circ\varphi$ a bounded composition operator
\[
\varphi^{\ast}:L^{1,p}(\widetilde{\Omega})\to L^{1,s}(\Omega),\,\,\,1< s\leq p<\infty,
\]
if and only if $\varphi\in W_{\loc}^{1,1}(\Omega)$, has finite distortion,
and
\begin{equation}
\label{kps}
\begin{aligned}
K_{p,p}(\varphi;\Omega) &= \ess\sup_\Omega 
   \left(\frac{|D\varphi(x)|^p}{|J(x,\varphi)|}\right)^{\tfrac{1}{p}} < \infty,
   \quad \text{if } 1 < s=p < \infty, \\
K_{p,s}(\varphi;\Omega) &= \left(\int_\Omega 
   \left(\frac{|D\varphi(x)|^p}{|J(x,\varphi)|}\right)^{\tfrac{s}{p-s}}\, dx\right)^{\tfrac{p-s}{ps}} < \infty,
   \quad \text{if } 1 < s < p < \infty.
\end{aligned}
\end{equation}

\end{thm}

Let us recall the formula of the change of variables in the Lebesgue integral \cite{Fe69, Ha93} in the required for us form.
Let $\Omega,\widetilde{\Omega}$ be domains in $\mathbb R^n$ and let a homeomorphism $\varphi : \Omega\to \widetilde{\Omega}$ be such that there exists a collection of closed sets $\{A_k\}_1^{\infty}$, $A_k\subset A_{k+1}\subset \Omega$ for which restrictions $\varphi \vert_{A_k}$ are Lipschitz mappings on the sets $A_k$ and 
$$
\biggl|\Omega\setminus\bigcup\limits_{k=1}^{\infty}A_k\biggr|=0.
$$
Then there exists a Borel set $S\subset \Omega$, $|S|=0$,  such that  the mapping $\varphi:\Omega\setminus S \to \widetilde{\Omega}\setminus\widetilde{S}$, $\widetilde{S}=\varphi(S)$, has the Luzin $N$-property and the change of variables formula
\begin{equation}
\label{chvf}
\int\limits_E f\circ\varphi (x) |J(x,\varphi)|~dx=\int\limits_{\varphi(E)\setminus\widetilde{S}} f(y)~dy
\end{equation}
holds for every measurable set $E\subset \Omega$ and for every nonnegative measurable function $f: \widetilde{\Omega}\to\mathbb R$. 
If a mapping $\varphi$ possesses the Luzin $N$-property (the image of a set of measure zero has measure zero), then $|\varphi (S)|=0$ and the second integral can be rewritten as the integral on $\varphi(E)$.

\section{Compactness of the trace operator}

\subsection{Sufficient conditions in outward cuspidal domains}

In this subsection, we establish conditions on weights which give compactness of the trace operator in the outward $\gamma$-cuspidal domain $\Omega_\gamma$. We focus on the behavior of the weight near the vertex of the outward $\gamma$-cusp.

\begin{thm}\label{main}
    Let $\Omega_{\gamma}\subset\mathbb R^n$ be a bounded outward $\gamma$-cuspidal domain and $W^{1,p}(\Omega_{\gamma})$, $1 < p<n<\gamma<\infty$, be a Sobolev space on $\Omega_\gamma$. Then the trace embedding operator
    $$
        T: W^{1,p}(\Omega_\gamma) \to L^q(\partial\Omega_\gamma, w_\gamma)
    $$
    is continuous for all $1 < q \leq \frac{p(n-1)}{n-p}$ with the weight 
\begin{equation}
\label{weight}
w_{\gamma}(x) = x_n^{\frac{(\gamma - n)(1+p(n-2))}{(n-p)(n-1)}},\,\, x=(x',x_n) \in \partial\Omega_\gamma,
\end{equation}

Moreover, the trace operator is compact if $1 < q < \frac{p(n-1)}{n-p}$.
\end{thm}

\begin{proof}
To prove this weighted trace theorem, we pull back the trace problem from the singular domain $\Omega_\gamma$ to the Lipschitz domain $\Omega_n$ by using mappings $\varphi_a : \Omega_n \to \Omega_\gamma$, $a>0$, which generate bounded composition operators
\[
\varphi_a^{\ast}: W^{1,p}(\Omega_\gamma) \to W^{1,p}(\Omega_n), \qquad 1 < p < n.
\]
Using these mappings, we obtain trace embeddings into weighted spaces $L^{q}(\partial\Omega_\gamma, w_{\gamma,a})$ ordered by inclusion. By selecting the optimal target space in this family, we complete the proof. This method is based on the following anticommutative diagram, suggested in \cite{GG94,GU09}:
    \begin{equation}\label{Diag}
\begin{CD}
    W^{1,p}(\Omega_\gamma) @>{\varphi_a^\ast}>> W^{1,p}(\Omega_n) \\
    @VVV @VVV \\
    \qquad L^q(\partial\Omega_\gamma, w_{\gamma,a}) 
        @<{(\varphi_a^{-1})^\ast}<< 
        L^q(\partial\Omega_n)
\end{CD}
\end{equation}

In this diagram, $\varphi_a^\ast$ and $(\varphi_a^{-1})^\ast$ are composition operators, defined by the rules $\varphi_a^\ast(f) = f \circ \varphi_a$ and $(\varphi_a^{-1})^\ast(g) = g \circ \varphi_a^{-1}$, respectively, where the homeomorphisms $\varphi_a : \overline{\Omega}_n \to \overline{\Omega}_\gamma$, $a>0$ (see, for example, \cite{GG94,GU09}), are given by
\[
    \varphi_a(y) = (y_1 y_n^{a\alpha - 1}, \dots, y_{n-1} y_n^{a\alpha - 1}, y_n^a), 
    \qquad \alpha = \frac{\gamma - 1}{n - 1}.
\]

The Jacobi matrix of $\varphi_a$ in $\Omega_n$ takes the form:
\begin{multline*}
    D\varphi_a(y) =
    \begin{pmatrix}
        y_n^{a\alpha - 1} & 0 & \dots & (a\alpha-1) y_1 y_n^{a\alpha - 2} \\
        0 & y_n^{a\alpha - 1} & \dots & (a\alpha-1) y_2 y_n^{a\alpha - 2} \\
        \dots & \dots & \dots & \dots \\
        0 & 0 & \dots & ay_n^{a-1}
    \end{pmatrix} \\
    = y_n^{a-1}
    \begin{pmatrix}
        y_n^{a\alpha - a} & 0 & \dots & (a\alpha-1)\frac{y_1}{y_n} y_n^{a\alpha - a} \\
        0 & y_n^{a\alpha - a} & \dots & (a\alpha-1)\frac{y_2}{y_n} y_n^{a\alpha - a} \\
        \dots & \dots & \dots & \dots \\
        0 & 0 & \dots & a
    \end{pmatrix}.
\end{multline*}

Since $0 < y_n < 1$ and $0 < \frac{y_1}{y_n} < 1$ for all $y \in \Omega_n$, we obtain the estimate
\[
    |D\varphi_a(y)| \leq y_n^{a-1} \bigl((a\alpha -1)^2(n-1) + (n-1) + a^2\bigr)^{1/2}.
\]
The Jacobian of this mapping is
\[
    J(y,\varphi_a) = a\, y_n^{(a\alpha-1)(n-1) + a - 1} = a\, y_n^{a\gamma - n}.
\]

Hence
\begin{multline*}
K_{p,p}(\varphi_a;\Omega_n)
= \ess\sup_{\Omega_n}
\left(\frac{|D\varphi_a(y)|^p}{|J(y,\varphi_a)|}\right)^{1/p} \\
\leq
\sup_{0<y_n<1}
\frac{y_n^{a-1} \bigl((n-1)((a\alpha -1)^2+1) + a^2\bigr)^{1/2}}
     {a^{1/p} y_n^{(a\gamma - n)/p}}
		\\
=
\sup_{0<y_n<1} y_n^{\,a-1 - \frac{a\gamma - n}{p}}
\left(\frac{1}{a}\right)^{1/p}
\bigl((n-1)((a\alpha -1)^2+1) + a^2\bigr)^{1/2}
< \infty,
\end{multline*}
provided that
\[
    a - 1 - \frac{a\gamma - n}{p} \ge 0
    \;\Longleftrightarrow\;
    a \le \frac{n-p}{\gamma - p}.
\]

Therefore, by Theorem~\ref{CompTh}, the homeomorphism $\varphi_a$, $a \in (0, \frac{n-p}{\gamma - p}]$, generates a bounded composition operator
\[
    \varphi_a^\ast : L^{1,p}(\Omega_\gamma) \to L^{1,p}(\Omega_n).
\]

The mapping $\varphi_a : \Omega_n \to \Omega_\gamma$ is a local bi-Lipschitz homeomorphism, and its inverse 
$\varphi_a^{-1} : \Omega_\gamma \to \Omega_n$ has the Jacobian
\[
    J(x,\varphi_a^{-1}) = \frac{1}{a}\, x_n^{\frac{n}{a}-\gamma},
\]
which is bounded provided that $\frac{n}{a}-\gamma \ge 0$, i.e. $a \le \frac{n}{\gamma}$.

Hence, by \cite{VU04,VU05}, the homeomorphism $\varphi_a$ generates a bounded composition operator
\[
    \varphi_a^\ast : L^{p}(\Omega_\gamma) \to L^{p}(\Omega_n),\quad 0 < a \le \frac{n}{\gamma},
\]
and so, the homeomorphism $\varphi_a$ generates a bounded composition operator
\[
    \varphi_a^\ast : W^{1,p}(\Omega_\gamma) \to W^{1,p}(\Omega_n)
\]
for all
\[
    0 < a \le \min\left\{\frac{n}{\gamma},\, \frac{n-p}{\gamma-p}\right\}
    = \frac{n-p}{\gamma-p}.
\]

It is known (see, e.g. \cite{M}) that in the Lipschitz domain $\Omega_n$ the trace operator
\[
    T : W^{1,p}(\Omega_n) \to L^q(\partial\Omega_n), \qquad 1 < p < n,
\]
is continuous for all $1 < q \le \frac{p(n-1)}{n-p}$ and is compact for all $1 < q < \frac{p(n-1)}{n-p}$.

Let us prove that the composition operator
\[
    (\varphi_a^{-1})^\ast : L^q(\partial\Omega_n) 
    \to L^q(\partial\Omega_\gamma, w_{\gamma,a}), 
    \qquad 
    w_{\gamma,a}(x) = x_n^{\frac{n-1}{a} - (n-2)\alpha - 1},
\]
acting on Lebesgue spaces with respect to the $(n-1)$-dimensional Hausdorff measure, is bounded.

The  restriction of $\varphi_a^{-1}$ to $\partial\Omega_\gamma$ is a locally bi-Lipschitz homeomorphism $\varphi_a^{-1}|_{\partial\Omega_\gamma}: \partial\Omega_{\gamma}\to \partial\Omega_{n}$.
Hence, by the area formula \cite{Fe69}, the equality 
    \begin{equation}\label{eq:area-boundary}
        \int_{\partial\Omega_n} g(y)\, d\mathcal H^{n-1}(y) = \int_{\partial\Omega_\gamma} g(\varphi_a^{-1}(x))\, 
				|J(x, \varphi_a^{-1}|_{\partial\Omega_\gamma})|\, d\mathcal H^{n-1}(x),
    \end{equation}
 for every measurable $g:\partial\Omega_n\to[0,\infty)$, where $J(x, \varphi_a^{-1}|_{\partial\Omega_\gamma})$ denotes the $(n-1)$--dimensional tangential Jacobian of $\varphi_a^{-1}$ restricted to $\partial\Omega_\gamma$.

Further we compute $J(x, \varphi_a^{-1}|_{\partial\Omega_\gamma})$ on the $(n-1)$-dimensional faces of the boundary $\partial\Omega_{\gamma}$. Recall that for a $C^1$--parametrization $\Phi:U\subset\mathbb R^{n-1}\to\partial\Omega_\gamma$, the tangential Jacobian of $\varphi_a^{-1}$ at a point $x=\Phi(u)$ is given by
$$
    J(x, \varphi_a^{-1}|_{\partial\Omega_\gamma}) = \sqrt{\det\bigl( D(\varphi_a^{-1}\circ\Phi)(u)^T\,D(\varphi_a^{-1}\circ \Phi)(u)\bigr)},
$$
and does not depend on the choice of the parametrization.

On boundary faces $F_i^0:=\{x\in\partial\Omega_\gamma:\ x_i=0\}$, $i=1,\dots,n-1$, we use coordinates $\hat{x} = (x_1,\dots,x_{i-1},x_{i+1},\dots,x_n)$. The tangential differential of $\varphi_a^{-1}$ with respect to these coordinates is upper triangular, hence
\begin{equation}\label{eq:Jpartial-flat}
    J(x, \varphi_a^{-1}|_{F_i^0}) = \frac{1}{a}\,x_n^{\frac{n-1}{a}-(n-2)\alpha-1},  \quad x\in F_i^0.
\end{equation}

The boundary faces $F_i^\alpha=\{x\in\partial\Omega_\gamma:\ x_i=x_n^\alpha\}$, $i=1,\dots,n-1$, are parametrized as
$$
    \Phi_i^\alpha(\hat x,t)=(x_1,\dots,x_{i-1},t^\alpha,x_{i+1},\dots,x_{n-1},t),
$$
where $\hat x=(x_1,\dots,x_{i-1},x_{i+1},\dots,x_{n-1})$ and $0<x_j<t^\alpha$ for $j\neq i$, and $0<t<1$.
Then 
\begin{equation}\label{eq:Jpartial-slanted-exact}
    J(x, \varphi_a^{-1}|_{F_i^\alpha}) = t^{(\frac1a-\alpha)(n-2)}\,t^{\frac1a-1}\, \sqrt{\frac{1}{a^2}+\alpha^2+(\tfrac1a-\alpha)^2\sum_{j\neq i}\frac{x_j^2}{t^{2\alpha}} }.
\end{equation}

Since $0<x_j<t^\alpha$ for $j\neq i$, we have $0\le x_j^2/t^{2\alpha}\le 1$ and hence
$$
    \frac{1}{a} \leq \sqrt{\frac{1}{a^2}+\alpha^2+(\tfrac1a-\alpha)^2\sum_{j\neq i}\frac{x_j^2}{t^{2\alpha}} } \leq \sqrt{\frac1{a^2}+\alpha^2+(\tfrac1a-\alpha)^2(n-2)}.
$$
Combining this inequality with \eqref{eq:Jpartial-slanted-exact} we have
\begin{equation}\label{eq:Jpartial-slanted-bounds}
    \frac{1}{a}\,x_n^{\frac{n-1}{a}-(n-2)\alpha-1} \leq J(x, \varphi_a^{-1}|_{F_i^\alpha}) \leq \sqrt{\frac1{a^2}+\alpha^2+(\tfrac1a-\alpha)^2(n-2)}\, x_n^{\frac{n-1}{a}-(n-2)\alpha-1}.
\end{equation}

Finally, taking the maximum of the upper constants from \eqref{eq:Jpartial-flat}
and \eqref{eq:Jpartial-slanted-bounds} we obtain
\begin{equation}\label{eq:tang-jac-bounds}
    \frac{1}{a}\,x_n^{\frac{n-1}{a}-(n-2)\alpha-1} \leq J(x, \varphi_a^{-1}|_{\partial\Omega_\gamma}) \leq C(a,n,\alpha)\,x_n^{\frac{n-1}{a}-(n-2)\alpha-1},
\end{equation}
with
$$
C(a,n,\alpha)= \sqrt{\frac1{a^2}+(\tfrac1a-\alpha)^2(n-1)+\alpha^2}.
$$

    In particular, the composition operator
    $$
        (\varphi_a^{-1})^*: L^q(\partial\Omega_n) \to L^q(\partial\Omega_\gamma, w_{\gamma,a}),
    $$
    is a bounded operator with $w_{\gamma, a} (x) = x_n^{\frac{n-1}{a}-(n-2)\alpha-1}$.

By the previous steps we conclude that the trace operator
\[
    T: W^{1,p}(\Omega_\gamma) \to L^q(\partial\Omega_\gamma, w_{\gamma,a})
\]
is continuous for all \(q \le \frac{p(n-1)}{n-p}\), and is compact for all \(q < \frac{p(n-1)}{n-p}\).

Thus we obtain the trace embedding of \(W^{1,p}(\Omega_\gamma)\) into the spaces \(L^q(\partial\Omega_\gamma, w_{\gamma,a})\) with weights
\[
w_{\gamma,a} = x_n^{\frac{n-1}{a} - (n-2)\alpha - 1}.
\]

Since the function \(\beta(a) = \frac{n-1}{a} - (n-2)\alpha - 1\) is strictly decreasing for \(a>0\), for \(a_1 < a_2\) we have \(\beta(a_2) < \beta(a_1)\), and hence, for \(0 < x_n < 1\),
\[
    x_n^{\beta(a_1)} \le x_n^{\beta(a_2)}.
\]
Therefore,
\[
w_{\gamma,a_1}(x) < w_{\gamma,a_2}(x).
\]
This implies that \(L^q(\partial\Omega_\gamma, w_{\gamma,a_2}) \subset L^q(\partial\Omega_\gamma, w_{\gamma,a_1})\) whenever \(a_1 < a_2\), see e.g. \cite{RS21}.

Thus the sharp target space among \(L^q(\partial\Omega_\gamma, w_{\gamma,a})\) is the one corresponding to the maximal value \(a = \frac{n-p}{\gamma-p}\), i.e. the sharp trace embedding is
\[
T: W^{1,p}(\Omega_\gamma) \hookrightarrow L^q(\partial\Omega_\gamma, w_\gamma),
\]
with the weight
\[
w_\gamma(x) = x_n^{\frac{(\gamma - n)(1 + p(n-2))}{(n-p)(n-1)}}.
\]

\end{proof}

Since the weight in Theorem~\ref{main} controls the boundary singularity of the cusp, we can now deduce the corresponding unweighted trace estimate.

\begin{cor}\label{cor:unweighted-trace}
Let $\Omega_{\gamma}\subset\mathbb R^n$ be a bounded outward $\gamma$-cuspidal domain and let
$1<p<n<\gamma<\infty$. Then the trace operator
$$
T: W^{1,p}(\Omega_\gamma)\to L^r(\partial\Omega_\gamma)
$$
is compact for all
$$
1<r<\frac{p\bigl(1+(n-2)\gamma\bigr)}{(\gamma-p)(n-1)}= \frac{p\, d(\gamma)}{n-p}
$$
where \(d(\gamma)=\frac{(n-p)(\gamma n - 2\gamma + 1)}{(\gamma-p)(n-1)}\) 
can be interpreted as the effective boundary dimension of the cusp 
(reducing to \(n-1\) in the Lipschitz case).
\end{cor}

\begin{proof}
Set
$$
p^*:=\frac{p(n-1)}{n-p},
\qquad
\alpha:=\frac{\gamma-1}{n-1},
\qquad
\beta:=\frac{(\gamma-n)(1+p(n-2))}{(n-p)(n-1)}.
$$
By Theorem~\ref{main}, for every
$$
1<q<p^*
$$
the trace operator
$$
T_q:W^{1,p}(\Omega_\gamma)\to L^q(\partial\Omega_\gamma,w_\gamma),
\qquad
w_\gamma(x)=x_n^\beta,
$$
is compact.

Fix $r$
$$
1<r<\frac{p\bigl(1+(n-2)\gamma\bigr)}{(\gamma-p)(n-1)}.
$$
We claim that one can choose \(q\) so that
$$
r<q<p^*
$$
and the embedding
$$
E:L^q(\partial\Omega_\gamma,w_\gamma)\hookrightarrow L^r(\partial\Omega_\gamma)
$$
is bounded.

Indeed, by H\"older's inequality, for every \(f\in L^q(\partial\Omega_\gamma,w_\gamma)\),
\begin{align*}
\|f\|_{L^r(\partial\Omega_\gamma)}^r
&=
\int_{\partial\Omega_\gamma} |f(x)|^r\, w_\gamma(x)^{\frac rq}\, w_\gamma(x)^{-\frac rq}\,d\mathcal H^{n-1}(x)
\\
&\le
\Big(\int_{\partial\Omega_\gamma} |f(x)|^q\, w_\gamma(x)\, d\mathcal H^{n-1}(x)\Big)^{\frac rq}
\Big(\int_{\partial\Omega_\gamma} w_\gamma(x)^{-\frac r{q-r}}\, d\mathcal H^{n-1}(x)\Big)^{\frac{q-r}{q}}
\\
&=
\|f\|_{L^q(\partial\Omega_\gamma,w_\gamma)}^r
\Big(\int_{\partial\Omega_\gamma} x_n^{-\beta\frac r{q-r}}\, d\mathcal H^{n-1}(x)\Big)^{\frac{q-r}{q}}.
\end{align*}
Thus it remains to verify the integrability condition
\begin{equation}\label{eq:need-integrability-corr}
\int_{\partial\Omega_\gamma} x_n^{-\beta\frac r{q-r}}\, d\mathcal H^{n-1}(x)<\infty.
\end{equation}

We use the decomposition of \(\partial\Omega_\gamma\) into the boundary faces
\[
F_i^0 := \{x \in \partial\Omega_\gamma : x_i = 0\},
\qquad
F_i^\alpha := \{x \in \partial\Omega_\gamma : x_i = x_n^\alpha\},
\qquad
i = 1,\dots,n-1,
\]
together with \(F^t:=\{x_n=1\}\).

The contribution of \(F^t\) is finite. On each flat face \(F_i^0\), using its parametrization, we obtain
\[
\int_{F_i^0} x_n^{-\beta\frac r{q-r}}\, d\mathcal H^{n-1}
=
\int_0^1 t^{\alpha(n-2)-\beta\frac r{q-r}}\,dt.
\]
On each slanted face \(F_i^\alpha\), again by the parametrization and the area formula,
\[
\int_{F_i^\alpha} x_n^{-\beta\frac r{q-r}}\, d\mathcal H^{n-1}
=
\int_0^1 t^{\alpha(n-2)-\beta\frac r{q-r}}
\sqrt{1+\alpha^2 t^{2\alpha-2}}\,dt
\le
C\int_0^1 t^{\alpha(n-2)-\beta\frac r{q-r}}\,dt,
\]
where \(C>0\) depends only on \(n\) and \(\alpha\).
Therefore \eqref{eq:need-integrability-corr} is satisfied whenever
\[
\alpha(n-2)-\beta\frac r{q-r}>-1.
\]
This is equivalent to
\[
q>r\Bigl(1+\frac{\beta}{\alpha(n-2)+1}\Bigr).
\]

Now
\[
\alpha(n-2)+1=\frac{1+(n-2)\gamma}{n-1},
\]
hence
\[
\frac{p^*(\alpha(n-2)+1)}{\alpha(n-2)+1+\beta}
=
\frac{p\bigl(1+(n-2)\gamma\bigr)}{(\gamma-p)(n-1)}.
\]
Therefore the assumption
\[
r<\frac{p\bigl(1+(n-2)\gamma\bigr)}{(\gamma-p)(n-1)}
\]
is equivalent to
\[
r\Bigl(1+\frac{\beta}{\alpha(n-2)+1}\Bigr)<p^*.
\]
Hence one can choose \(q\) such that
\[
r\Bigl(1+\frac{\beta}{\alpha(n-2)+1}\Bigr)<q<p^*.
\]
For this choice of \(q\), the embedding
\[
E:L^q(\partial\Omega_\gamma,w_\gamma)\hookrightarrow L^r(\partial\Omega_\gamma)
\]
is bounded.

Since \(T_q:W^{1,p}(\Omega_\gamma)\to L^q(\partial\Omega_\gamma,w_\gamma)\) is compact and \(E\) is bounded, their composition
\[
T=E\circ T_q:W^{1,p}(\Omega_\gamma)\to L^r(\partial\Omega_\gamma)
\]
is compact.

Finally,
\[
\frac{p\bigl(1+(n-2)\gamma\bigr)}{(\gamma-p)(n-1)}
=
\frac{p(d(\gamma))}{n-p},
\]
where
\[
d(\gamma)=\frac{(n-p)(\gamma n - 2\gamma + 1)}{(\gamma-p)(n-1)}.
\]
This completes the proof.
\end{proof}

\subsection{Comparison with the previous results}

In \cite{GV}, it was proved that in the case of outward $\gamma$-cuspidal domains
$$
    \widetilde{\Omega}_{\gamma} = \{x=(x_1, \dots, x_n) \in \mathbb{R}^n: \sqrt{x_1^2+\dots+x_{n-1}^2} < x_n^{\alpha}, 0<x_n<1\}, \quad \alpha = \frac{\gamma-1}{n-1},
$$
the trace operator
$$
\widetilde{T}: W^{1,p}(\widetilde{\Omega}_{\gamma}) \to L^p(\partial\widetilde{\Omega}_{\gamma}, \widetilde{w}_{\gamma})
$$
is compact with  the weight function $\widetilde{w}_{\gamma} = x_n^{\alpha}$.

With this in mind, Theorem~\ref{main} of the present work implies that the trace operator
$$
    T: W^{1,p}(\Omega_\gamma) \to L^q(\partial\Omega_\gamma, w_\gamma)
$$
is compact with the weight function $w_{\gamma}(x) = x_n^\beta$, $\beta:=\frac{(\gamma - n)(1+p(n-2))}{(n-p)(n-1)}$.

Let us show that the compactness of the trace operator $T$ implies the compactness of the trace operator 
$\widetilde{T}$. Since outward $\gamma$-cuspidal domains $\widetilde{\Omega}_{\gamma}$ and $\Omega_\gamma$ are bi-Lipschitz equivalent, it is sufficient to prove that, for a suitable \(q\), there is a bounded embedding 
$$
L^q(\partial\Omega_\gamma, w_\gamma)\hookrightarrow L^p(\partial\Omega_\gamma, \widetilde{w}_\gamma).
$$

Let \(q>p\). Using the H\"older inequality, we obtain
\begin{multline*}
    \|f\|_{L^p(\partial\Omega_\gamma,\widetilde w)}^p
    = \int_{\partial\Omega_\gamma} |f|^p x_n^\alpha\,d\mathcal H^{n-1} \\
    \leq \Bigl(\int_{\partial\Omega_\gamma} |f|^q x_n^{\beta}\,d\mathcal H^{n-1}\Bigr)^{p/q}
    \Bigl(\int_{\partial\Omega_\gamma} x_n^{\bigl(\alpha-\frac{p}{q}\beta\bigr)\frac{q}{q-p}}\,d\mathcal H^{n-1}\Bigr)^{\frac{q-p}{q}}.
\end{multline*}

Hence it remains to verify the finiteness of
$$
\int_{\partial\Omega_\gamma} x_n^{\theta}\,d\mathcal H^{n-1},
\qquad
\theta:=\Bigl(\alpha-\frac{p}{q}\beta\Bigr)\frac{q}{q-p}.
$$

Using the parametrizations of the boundary faces \(F_i^0\), \(F_i^\alpha\), and \(F^t\) as before,
we see that this integral is finite if and only if
$$
\theta+\alpha(n-2)>-1.
$$
A direct computation shows that this condition is equivalent to
$$
q>\frac{p(n-1)(\gamma-p)}{\gamma(n-p)}.
$$

Since
$$
\frac{p(n-1)(\gamma-p)}{\gamma(n-p)}
<
\frac{p(n-1)}{n-p},
$$
we can choose \(q\) such that
$$
\frac{p(n-1)(\gamma-p)}{\gamma(n-p)}<q<\frac{p(n-1)}{n-p}.
$$
For this choice of \(q\), the embedding
$$
L^q(\partial\Omega_\gamma, w_\gamma)\hookrightarrow L^p(\partial\Omega_\gamma, \widetilde{w}_\gamma)
$$
is bounded.

Thus Theorem~\ref{main} implies the trace embedding result of \cite{GV}.

\subsection{Sharpness of trace embeddings with power weights}

In this subsection we consider trace embedding operators with power weights
\[
w_\theta(x):=x_n^\theta,
\qquad x=(x',x_n)\in \partial\Omega_\gamma.
\]
Recall the notation
\[
\alpha=\frac{\gamma-1}{n-1},
\qquad
p^*=\frac{p(n-1)}{n-p},
\qquad
\beta=\frac{(\gamma-n)(1+p(n-2))}{(n-p)(n-1)}.
\]

The next theorem provides a necessary condition on the power exponent~$\theta$ for the boundedness of the trace embedding into power-weighted Lebesgue spaces.

\begin{thm}
\label{prop:necessary-power-weight}
Let 
\[
T:W^{1,p}(\Omega_\gamma)\to L^q(\partial\Omega_\gamma,x_n^\theta),
\]
with \(1<p<n\) and \(1<q<\infty\), be a bounded trace embedding operator. 
Then the exponent \(\theta\) satisfies
\begin{equation}\label{eq:theta-necessary-general}
\theta \ge \frac{q}{p}\bigl(\alpha(n-1)+1-p\bigr)-\alpha(n-2)-1.
\end{equation}
In particular, for the critical exponent \(q=p^*=\frac{p(n-1)}{n-p}\), one has
\begin{equation}\label{eq:theta-necessary-critical}
\theta\ge \beta.
\end{equation}
\end{thm}

\begin{proof}
Recall the decomposition of \(\partial\Omega_\gamma\) into the boundary faces
\[
F_i^0 := \{x \in \partial\Omega_\gamma : x_i = 0\},
\qquad
F_i^\alpha := \{x \in \partial\Omega_\gamma : x_i = x_n^\alpha\},
\qquad
i = 1,\dots,n-1,
\]
together with \(F^t:=\{x_n=1\}\) which is irrelevant for the argument.

Fix a cut-off function \(\eta\in C^\infty([0,\infty))\), \(0\le \eta\le 1\), such that  
\(\eta\equiv 1\) on \([0,1]\) and \(\eta\equiv 0\) on \([2,\infty)\). 
For every \(0<\varepsilon<\tfrac12\) define the test function
\[
u_\varepsilon(x):=\eta\!\left(\frac{x_n}{\varepsilon}\right),
\qquad x\in\Omega_\gamma,
\]
and denote \(f_\varepsilon:=T(u_\varepsilon)\).
Then \(u_\varepsilon\in W^{1,p}(\Omega_\gamma)\), and  
\(f_\varepsilon\equiv 1\) on \(\partial\Omega_\gamma\cap\{0<x_n<\varepsilon\}\).

Since the trace operator 
\[
T:W^{1,p}(\Omega_\gamma)\to L^q(\partial\Omega_\gamma,x_n^\theta),
\]
is bounded, there exists a constant \(C_{p,q}<\infty\) such that
\begin{equation}\label{eq:trace-ineq-test-power}
\|f_\varepsilon\|_{L^q(\partial\Omega_\gamma,x_n^\theta)}
\le
C_{p,q}\|u_\varepsilon\|_{W^{1,p}(\Omega_\gamma)},
\qquad \text{for all } 0<\varepsilon<\frac{1}{2}.
\end{equation}

We begin by estimating the left-hand side of \eqref{eq:trace-ineq-test-power} from below.

Consider boundary faces \(F_i^0\), \(i=1,\dots,n-1\). Fix a number \(i\in\{1,\dots,n-1\}\) and parametrize \(F_i^0\) by
\[
\Phi_i^0(\hat x,t)
=
(x_1,\dots,x_{i-1},0,x_{i+1},\dots,x_{n-1},t),\,\, 0<x_j<t^\alpha\,\,\text{if}\,\,j\neq i,\,\, 0<t<1,
\]
where \(\hat x=(x_1,\dots,x_{i-1},x_{i+1},\dots,x_{n-1})\).

Using the parametrization of \(F_i^0\), the surface measure satisfies  
\[
d\mathcal H^{n-1} = t^{\alpha(n-2)}\, d\hat x\, dt.
\]
Since \(f_\varepsilon\equiv 1\) on \(\partial\Omega_\gamma\cap\{0<x_n<\varepsilon\}\), we obtain
\begin{multline}
\|f_\varepsilon\|_{L^q(\partial\Omega_\gamma,x_n^\theta)}^q
\ge
\int_{F_i^0\cap\{0<x_n<\varepsilon\}} x_n^\theta\, d\mathcal H^{n-1}
\\
=
\int_0^\varepsilon t^\theta \!\!\int_{(0,t^\alpha)^{n-2}} t^{\alpha(n-2)}\, d\hat x\, dt
=
\int_0^\varepsilon t^{\theta+\alpha(n-2)}\, dt < \infty,
\label{eq:LHS-lower-flat}
\end{multline}
if \(\theta+\alpha(n-2)>-1\).

Next consider boundary faces \(F_i^\alpha\), \(i=1,\dots,n-1\). 
Parameterizing \(F_i^\alpha\) by
\[
\Phi_i^\alpha(\hat x,t)
=
(x_1,\dots,x_{i-1},t^\alpha,x_{i+1},\dots,x_{n-1},t),
\qquad
0 < x_j < t^\alpha \ \text{if } j\neq i,\ \ 0<t<1,
\]
the surface measure satisfies
\[
d\mathcal H^{n-1}
=
\sqrt{1+\alpha^2 t^{2\alpha-2}}\; t^{\alpha(n-2)}\, d\hat x\, dt.
\]

Since \(f_\varepsilon\equiv 1\) on \(\partial\Omega_\gamma\cap\{0<x_n<\varepsilon\}\), we obtain
\begin{multline}
\|f_\varepsilon\|_{L^q(\partial\Omega_\gamma,x_n^\theta)}^q
\ge
\int_{F_i^\alpha\cap\{0<x_n<\varepsilon\}} x_n^\theta\, d\mathcal H^{n-1}
\\
=
\int_0^\varepsilon t^\theta \!\!\int_{(0,t^\alpha)^{n-2}}
\sqrt{1+\alpha^2 t^{2\alpha-2}}\; t^{\alpha(n-2)}\, d\hat x\, dt
\\
\ge
\int_0^\varepsilon t^\theta \!\!\int_{(0,t^\alpha)^{n-2}} t^{\alpha(n-2)}\, d\hat x\, dt
=
\int_0^\varepsilon t^{\theta+\alpha(n-2)}\, dt < \infty,
\label{eq:LHS-lower-slanted}
\end{multline}
if \(\theta+\alpha(n-2)>-1\).

Hence
\begin{equation}
\label{eq:LHS-eps-power}
\|f_\varepsilon\|_{L^q(\partial\Omega_\gamma,x_n^\theta)}
\ge
C\,\varepsilon^{(\theta+\alpha(n-2)+1)/q},
\qquad \text{if } \theta+\alpha(n-2)>-1.
\end{equation}

Next we estimate the right-hand side of \eqref{eq:trace-ineq-test-power}. 
Since \(u_\varepsilon\) depends only on \(x_n\), its gradient satisfies
\[
\nabla u_\varepsilon(x)=0 
\quad\text{for } x_n\notin(\varepsilon,2\varepsilon),
\qquad
|\nabla u_\varepsilon(x)|\le C\,\varepsilon^{-1}
\quad\text{for } x_n\in(\varepsilon,2\varepsilon).
\]

Hence
\[
\|u_\varepsilon\|_{L^p(\Omega_\gamma)}^p
\le
\int_{\Omega_\gamma\cap\{0<x_n<2\varepsilon\}} 1\,dx
\le
\int_0^{2\varepsilon} t^{\alpha(n-1)}\,dt
\le
C\,\varepsilon^{\alpha(n-1)+1},
\]
and
\[
\|\nabla u_\varepsilon\|_{L^p(\Omega_\gamma)}^p
\le
\int_{\Omega_\gamma\cap\{\varepsilon<x_n<2\varepsilon\}} \varepsilon^{-p}\,dx
\le
\varepsilon^{-p}\int_\varepsilon^{2\varepsilon} t^{\alpha(n-1)}\,dt
\le
C\,\varepsilon^{\alpha(n-1)+1-p}.
\]
Therefore,
\begin{align}
\|u_\varepsilon\|_{W^{1,p}(\Omega_\gamma)}
&=
\|\nabla u_\varepsilon\|_{L^p(\Omega_\gamma)}
+
\|u_\varepsilon\|_{L^p(\Omega_\gamma)} \notag\\
&\le
C\,\varepsilon^{(\alpha(n-1)+1-p)/p}
+
C\,\varepsilon^{(\alpha(n-1)+1)/p}
\notag\\
&\le
C\,\varepsilon^{(\alpha(n-1)+1-p)/p},
\label{eq:RHS-eps2-power}
\end{align}
where a constant \(C\) depends on \(n\) and \(\alpha\) only.

Substituting estimates \eqref{eq:LHS-eps-power} and \eqref{eq:RHS-eps2-power} 
in the inequality \eqref{eq:trace-ineq-test-power} we obtain
\[
\varepsilon^{(\theta+\alpha(n-2)+1)/q}
\le
C\,\varepsilon^{(\alpha(n-1)+1-p)/p}
\qquad \text{for all } 0<\varepsilon<\tfrac12.
\]

Hence
\[
\frac{\theta+\alpha(n-2)+1}{q}
\ge
\frac{\alpha(n-1)+1-p}{p},
\]
and so
\[
\theta
\ge
\frac{q}{p}\bigl(\alpha(n-1)+1-p\bigr)-\alpha(n-2)-1.
\]
The inequality \eqref{eq:theta-necessary-general} is proved.

Finally, substituting \(q=p^*=\frac{p(n-1)}{\,n-p\,}\) into
\eqref{eq:theta-necessary-general}, we obtain
\[
\theta \ge \beta
=
\frac{(\gamma-n)\bigl(1+p(n-2)\bigr)}{(n-p)(n-1)}.
\]

\end{proof}

\section{Weighted Schr\"odinger--Steklov $(p,q)$-eigenvalue problem}

Let $\Omega_{\gamma} \subset \R^n$ be an outward cuspidal domain defined by (\ref{cusp}), and let the weight function $w_{\gamma}$ be given by (\ref{weight}). Using the compactness of the Sobolev trace operator we consider the non-reduced weighted Schr\"odinger--Steklov $(p,q)$-eigenvalue problem, for $1 < p < n$ and $p<q<\tfrac{p(n-1)}{n-p}$,
\begin{equation}
\label{pqwSteklov}
\begin{cases}
-\mathrm{div}(|\nabla u|^{p-2}\nabla u) + |u|^{p-2}u = 0 & \text{in } \Omega_{\gamma},\\
|\nabla u|^{p-2} \nabla u \cdot \nu = \lambda\, w_{\gamma}\,\|u\|^{p-q}_{L^{q}(\partial \Omega_{\gamma}, w_{\gamma})}\, |u|^{q-2}u & \text{on } \partial\Omega_{\gamma}.
\end{cases}
\end{equation}

Define $(\lambda,u)\in \mathbb{R} \times (W^{1,p}(\Omega_{\gamma})\setminus\{0\})$ 
as an eigenpair of \eqref{pqwSteklov} if for every $v\in W^{1,p}(\Omega_{\gamma})$ we have
\begin{equation}
\label{weak}
\int_{\Omega_{\gamma}} |\nabla u|^{p-2} \nabla u \cdot \nabla v\,dx
+ \int_{\Omega_{\gamma}} |u|^{p-2} u\, v\,dx
= \lambda \|u\|^{p-q}_{L^q(\partial\Omega_{\gamma}, w_{\gamma})}\int_{\partial \Omega_{\gamma}} |u|^{q-2} u\, v\, w_{\gamma}\,d\mathcal H^{n-1}(x).
\end{equation}
We refer to $\lambda$ as an eigenvalue and $u$ as an eigenfunction of \eqref{pqwSteklov} corresponding to the eigenvalue $\lambda$. Denote by
\[
\Sigma_{p,q}(\Omega_\gamma,w_\gamma) =
\left\{
\lambda\in\mathbb R:
\exists\,u\in W^{1,p}(\Omega_\gamma)\setminus\{0\}
\text{ such that }(\lambda,u)\text{ is an eigenpair of \eqref{pqwSteklov}}
\right\}
\]
the set of eigenvalues of problem \eqref{pqwSteklov} and denote by
\[
\lambda_{p,q} = \lambda_{p,q}(\Omega_\gamma,w_\gamma) = \inf \Sigma_{p,q}(\Omega_\gamma,w_\gamma)
\]
the least (first) eigenvalue of problem \eqref{pqwSteklov}.

Note also, that the equation \eqref{pqwSteklov} represents the Euler-Lagrange equation corresponding, in its weak formulation \eqref{weak}, to the functional
$$
E(v) = \|\nabla v\|_{L^p(\Omega_{\gamma})}^p + \|v\|_{L^p(\Omega_{\gamma})}^p,
$$
restricted to the set
$$
S = \left\{ u \in W^{1,p}(\Omega_{\gamma}) : \|u\|_{L^q(\partial \Omega_{\gamma},w_{\gamma})} = 1 \right\}.
$$

The following theorem shows that the least eigenvalue
\(\lambda_{p,q}\) of problem \eqref{pqwSteklov} is positive, is attained
by an eigenpair, and admits a variational characterization as the minimum
of the corresponding Rayleigh quotient.

\begin{thm}
\label{thm:first-eigenvalue}
Let $\Omega_{\gamma}\subset\mathbb{R}^n$ be an outward cuspidal domain defined by \eqref{cusp}, and let the weight $w_{\gamma}$ be given by \eqref{weight}. Assume that $1<p<n$,$p<q<\frac{p(n-1)}{n-p}$.
Then the least eigenvalue \(\lambda_{p,q}\) is attained, i.e. there exists
\[
u_{p,q}\in W^{1,p}(\Omega_{\gamma})\setminus\{0\}
\]
such that \((\lambda_{p,q},u_{p,q})\) is an eigenpair of \eqref{pqwSteklov}.
In particular, \(\lambda_{p,q}>0\).

Moreover,
\begin{equation}\label{eq:Rayleigh}
\lambda_{p,q}
=
\min_{\substack{u\in W^{1,p}(\Omega_{\gamma})\\ Tu\not\equiv 0}}
\frac{\displaystyle \int_{\Omega_\gamma}\bigl(|\nabla u|^{p}+|u|^{p}\bigr)\,dx}
{\|u\|_{L^{q}(\partial\Omega_{\gamma},w_{\gamma})}^{\,p}}.
\end{equation}
\end{thm}

\begin{proof}
Consider two functionals
\[
E(u)=\int_{\Omega_\gamma} \bigl(|\nabla u|^p+|u|^p\bigr)\,dx,
\qquad
B(u)=\int_{\partial\Omega}|u|^q\,w_\gamma\,d\mathcal H^{n-1}
=
\|Tu\|_{L^q(\partial\Omega_\gamma,w_\gamma)}^q.
\]
The functional \(B\) is well defined by the continuity of the trace operator
\[
T:W^{1,p}(\Omega_\gamma)\to L^q(\partial\Omega_\gamma,w_\gamma).
\]

Consider the constraint set
\[
S=
\{u\in W^{1,p}(\Omega_\gamma): B(u)=1\}
=
\{u\in W^{1,p}(\Omega_\gamma): \|Tu\|_{L^q(\partial\Omega_\gamma,w_\gamma)}=1\}.
\]

For every \(u\in W^{1,p}(\Omega_\gamma)\) with \(Tu\not\equiv 0\), define
\[
\widetilde u=\frac{u}{\|u\|_{L^q(\partial\Omega_\gamma,w_\gamma)}}.
\]
Then \(\widetilde u\in S\), and by the \(p\)-homogeneity of \(E\),
\[
E(\widetilde u)
=
\frac{E(u)}{\|u\|_{L^q(\partial\Omega_\gamma,w_\gamma)}^p}.
\]
Therefore
\[
\inf_{\substack{u\in W^{1,p}(\Omega_\gamma)\\ Tu\not\equiv 0}}
\frac{E(u)}{\|u\|_{L^q(\partial\Omega_\gamma,w_\gamma)}^p}
=
\inf_{v\in S} E(v).
\]

Since \(1<p<\infty\), the map \(\xi\mapsto |\xi|^p\)
are \(C^1\) and convex. Hence the functional $E$
is \(C^1\), strictly convex, and weakly lower semicontinuous on
\(W^{1,p}(\Omega_\gamma)\). Moreover,
\[
E(u)=\|\nabla u\|_{L^p(\Omega_\gamma)}^p+\|u\|_{L^p(\Omega_\gamma)}^p
\ge 2^{1-p}\|u\|_{W^{1,p}(\Omega_\gamma)}^p,
\]
so \(E\) is coercive. 

If \(u_k\rightharpoonup u\)
weakly in \(W^{1,p}(\Omega_\gamma)\), then by compactness of the trace operator,
\[
Tu_k\to Tu
\qquad\text{strongly in } L^q(\partial\Omega_\gamma,w_\gamma),
\]
hence \(B(u_k)\to B(u)\). Thus \(S\) is weakly closed in
\(W^{1,p}(\Omega_\gamma)\).

By the direct method of the calculus of variations, there exists
\(u_*\in S\) such that infimum for $E$ is attained
\[
E(u_*)=\inf_{v\in S}E(v).
\]

The G\^ateaux derivative of \(E\) is given by
\[
\langle E'(u),v\rangle
=
p\int_{\Omega_\gamma}
\bigl(
|\nabla u|^{p-2}\nabla u\cdot\nabla v
+
|u|^{p-2}uv
\bigr)\,dx,
\qquad u,v\in W^{1,p}(\Omega_\gamma),
\]
and the G\^ateaux derivative of \(B\) is
\[
\langle B'(u),v\rangle
=
q\int_{\partial\Omega_\gamma}|u|^{q-2}u\,v\,w_\gamma\,d\mathcal H^{n-1},
\qquad u,v\in W^{1,p}(\Omega_\gamma).
\]
Since
\[
\langle B'(u_*),u_*\rangle = q B(u_*) = q \neq 0,
\]
the Lagrange multiplier rule applies. Hence there exists \(\mu\in\mathbb R\)
such that
\[
\langle E'(u_*),v\rangle=\mu\,\langle B'(u_*),v\rangle
\qquad\text{for all } v \in W^{1,p}(\Omega_\gamma).
\]

Taking \(v=u_*\), we obtain
\[
pE(u_*)=\mu q\,B(u_*)=\mu q,
\]
because \(u_*\in S\). 
Therefore,
\begin{equation}\label{eq:weak-normalized}
\int_{\Omega_\gamma}
\bigl( |\nabla u_*|^{p-2}\nabla u_*\cdot\nabla v + |u_*|^{p-2}u_*v
\bigr)\,dx
=
\frac{\mu q}{p}
\int_{\partial\Omega_\gamma}|u_*|^{q-2}u_*\,v\,w_\gamma\,d\mathcal H^{n-1}
\end{equation}
for all \(v\in W^{1,p}(\Omega_\gamma)\).

Hence \eqref{eq:weak-normalized} coincides with \eqref{weak}, so
\((\frac{\mu q}{p},u_*)\) is an eigenpair of \eqref{pqwSteklov}. In particular,
\(\frac{\mu q}{p}\in\Sigma_{p,q}(\Omega_\gamma,w_\gamma)\), and therefore
\[
\lambda_{p,q}
=
\inf\Sigma_{p,q}(\Omega_\gamma,w_\gamma)
\le \frac{\mu q}{p}.
\]

Conversely, let \(\lambda\in\Sigma_{p,q}(\Omega_\gamma,w_\gamma)\), and let
\(u\in W^{1,p}(\Omega_\gamma)\setminus\{0\}\) be a corresponding eigenfunction.
Taking \(v=u\) in \eqref{weak}, we obtain
\[
\int_{\Omega_\gamma}\bigl(|\nabla u|^p+|u|^p\bigr)\,dx
=
\lambda\,\|u\|_{L^q(\partial\Omega_\gamma,w_\gamma)}^{p-q}
\int_{\partial\Omega_\gamma}|u|^q\,w_\gamma\,d\mathcal H^{n-1}
=
\lambda\,\|u\|_{L^q(\partial\Omega_\gamma,w_\gamma)}^{p}.
\]
Since \(u\neq 0\), the left-hand side is strictly positive. Hence
\(\|u\|_{L^q(\partial\Omega_\gamma,w_\gamma)}>0\) and
\[
\lambda = \frac{E(u)}{\|u\|_{L^q(\partial\Omega_\gamma,w_\gamma)}^p}
\ge \frac{\mu q}{p}.
\]
As this holds for every \(\lambda\in\Sigma_{p,q}(\Omega_\gamma,w_\gamma)\), we get
\[
\lambda_{p,q}
=
\inf\Sigma_{p,q}(\Omega_\gamma,w_\gamma)
\ge \frac{\mu q}{p}.
\]
Combining the two inequalities, we conclude that
\[
\lambda_{p,q}=\frac{\mu q}{p}.
\]
Thus \((\lambda_{p,q},u_*)\) is an eigenpair, and \eqref{eq:Rayleigh} follows.
Setting \(u_{p,q}:=u_*\), the proof is complete.
\end{proof}

The variational characterization of the least eigenvalue obtained in
Theorem~\ref{thm:first-eigenvalue} immediately identifies it with the reciprocal
\(p\)-th power of the optimal weighted trace constant. In the cuspidal setting,
this constant can be estimated explicitly by combining the factorization of the
trace operator from Theorem~\ref{main} with the corresponding trace inequality
on the model Lipschitz domain \(\Omega_n\).

\begin{thm}[Optimal trace constant]\label{thm:trace-constant}
Under the assumptions of Theorem~\ref{thm:first-eigenvalue}, let
\(C_{\mathrm{tr}}(\Omega_\gamma)\) be the best constant in the weighted trace inequality
\begin{equation}\label{eq:trace-ineq}
  \|Tu\|_{L^q(\partial\Omega_\gamma,w_\gamma)}
  \leq
  C_{\mathrm{tr}}(\Omega_\gamma)
  \Biggl(\int_{\Omega_\gamma} \bigl(|\nabla u|^p + |u|^p\bigr)\,dx\Biggr)^{1/p},
  \qquad u\in W^{1,p}(\Omega_\gamma).
\end{equation}
Then
\begin{equation}\label{eq:best-constant}
  \lambda_{p,q}(\Omega_\gamma,w_\gamma)^{-1/p}
  =
  C_{\mathrm{tr}}(\Omega_\gamma).
\end{equation}

Moreover, in the cuspidal domain \(\Omega_\gamma\), the following estimate holds:
\begin{equation}\label{eq:ctr-cusp-estimate-explicit}
C_{\mathrm{tr}}(\Omega_\gamma)
\le
\left(\frac{n-p}{\gamma-p}\right)^{\frac{1}{q}-\frac{1}{p}}
\left((n-1)+\frac{(n-p)^2}{(\gamma-p)^2}
+\frac{(p-1)^2(\gamma-n)^2}{(\gamma-p)^2(n-1)}\right)^{1/2}
\,C_{\mathrm{tr}}(\Omega_n).
\end{equation}
\end{thm}

\begin{proof}
By Theorem~\ref{thm:first-eigenvalue},
\[
\lambda_{p,q}(\Omega_\gamma,w_\gamma)
=
\inf_{\substack{u\in W^{1,p}(\Omega_\gamma)\\ Tu\not\equiv 0}}
\frac{\int_{\Omega_\gamma}\bigl(|\nabla u|^p+|u|^p\bigr)\,dx}{\|Tu\|_{L^q(\partial\Omega_\gamma,w_\gamma)}^p}.
\]
which immediately proves \eqref{eq:best-constant}.

Now let \(\varphi:\overline{\Omega}_n\to\overline{\Omega}_\gamma\) be
\[
\varphi(y)
=
(y_1 y_n^{a\alpha - 1}, \dots, y_{n-1} y_n^{a\alpha - 1}, y_n^a),
\qquad
\alpha=\frac{\gamma-1}{n-1},
\qquad
a=\frac{n-p}{\gamma-p}.
\]
By the anticommutative diagram from Theorem~\ref{main}, the weighted trace operator
on \(\Omega_\gamma\) factorizes through the Lipschitz domain \(\Omega_n\):
\[
T_{\Omega_\gamma}
=
(\varphi^{-1})^* \circ T_{\Omega_n}\circ \varphi^*,
\]
where
\[
\varphi^*:W^{1,p}(\Omega_\gamma)\to W^{1,p}(\Omega_n),
\qquad
(\varphi^{-1})^*:L^q(\partial\Omega_n)\to L^q(\partial\Omega_\gamma,w_\gamma).
\]

From the lower bound in \eqref{eq:tang-jac-bounds},
\[
\frac1a\,w_\gamma(x)\le J(x,\varphi^{-1}|_{\partial\Omega_\gamma}),
\]
hence for every \(g\in L^q(\partial\Omega_n)\),
\begin{align*}
\|(\varphi^{-1})^*g\|_{L^q(\partial\Omega_\gamma,w_\gamma)}^q
&=
\int_{\partial\Omega_\gamma}|g\circ\varphi^{-1}|^q\,w_\gamma\,d\mathcal H^{n-1}\\
&\le
a\int_{\partial\Omega_\gamma}|g\circ\varphi^{-1}|^q
J(x,\varphi^{-1}|_{\partial\Omega_\gamma})\,d\mathcal H^{n-1}(x)\\
&=
a\int_{\partial\Omega_n}|g(y)|^q\,d\mathcal H^{n-1}(y).
\end{align*}
Therefore,
\[
\|(\varphi^{-1})^*\|\le a^{1/q}.
\]

Further, by Theorem~\ref{CompTh},
\[
\|\nabla(u\circ\varphi)\|_{L^p(\Omega_n)}
\le
K_p(\varphi;\Omega_n)\,\|\nabla u\|_{L^p(\Omega_\gamma)},
\]
while the change-of-variables formula and the boundedness of \(J(\cdot,\varphi^{-1})\) give
\[
\|u\circ\varphi\|_{L^p(\Omega_n)}
\le
\|J(\cdot,\varphi^{-1})\|_{L^\infty(\Omega_\gamma)}^{1/p}
\|u\|_{L^p(\Omega_\gamma)}.
\]
Hence
\begin{multline*}
\Biggl(\int_{\Omega_n}\bigl(|\nabla(u\circ\varphi)|^p+|u\circ\varphi|^p\bigr)\,dy\Biggr)^{1/p}\\
\le
\max\Bigl\{K_p(\varphi;\Omega_n),\,
\|J(\cdot,\varphi^{-1})\|_{L^\infty(\Omega_\gamma)}^{1/p}\Bigr\}
\Biggl(\int_{\Omega_\gamma}\bigl(|\nabla u|^p+|u|^p\bigr)\,dx\Biggr)^{1/p}.
\end{multline*}

Therefore, for every \(u\in W^{1,p}(\Omega_\gamma)\),
\begin{multline*}
\|T_{\Omega_\gamma}u\|_{L^q(\partial\Omega_\gamma,w_\gamma)}=
\|(\varphi^{-1})^*\bigl(T_{\Omega_n}(u\circ\varphi)\bigr)\|_{L^q(\partial\Omega_\gamma,w_\gamma)} \le
a^{1/q}\,\|T_{\Omega_n}(u\circ\varphi)\|_{L^q(\partial\Omega_n)}\\
\le a^{1/q}\,C_{\mathrm{tr}}(\Omega_n)
\Biggl(\int_{\Omega_n}\bigl(|\nabla(u\circ\varphi)|^p+|u\circ\varphi|^p\bigr)\,dy\Biggr)^{1/p}\\
\le a^{1/q} \max\Bigl\{K_p(\varphi;\Omega_n),\,
\|J(\cdot,\varphi^{-1})\|_{L^\infty(\Omega_\gamma)}^{1/p}\Bigr\}
C_{\mathrm{tr}}(\Omega_n)\times\\
\times\Biggl(\int_{\Omega_\gamma}\bigl(|\nabla u|^p+|u|^p\bigr)\,dx\Biggr)^{1/p}.
\end{multline*}
Consequently,
\[
C_{\mathrm{tr}}(\Omega_\gamma)
\le
a^{1/q}
\max\Bigl\{K_p(\varphi;\Omega_n),\,
\|J(\cdot,\varphi^{-1})\|_{L^\infty(\Omega_\gamma)}^{1/p}\Bigr\}
\,C_{\mathrm{tr}}(\Omega_n).
\]

By the direct computations from the proof of Theorem~\ref{main},
\[
K_p(\varphi;\Omega_n)
\le
a^{-1/p}
\Bigl((n-1)\bigl((a\alpha-1)^2+1\bigr)+a^2\Bigr)^{1/2},
\qquad
J(x,\varphi^{-1})=\frac1a\,x_n^{\,n/a-\gamma}.
\]
Since
\[
\frac{n}{a}-\gamma
=
\frac{n(\gamma-p)}{n-p}-\gamma
=
\frac{p(\gamma-n)}{n-p}>0,
\]
we obtain
\[
\|J(\cdot,\varphi^{-1})\|_{L^\infty(\Omega_\gamma)}^{1/p}
=
a^{-1/p}.
\]
Hence
\begin{align*}
\max\Bigl\{K_p(\varphi;\Omega_n),\,
\|J(\cdot,\varphi^{-1})\|_{L^\infty(\Omega_\gamma)}^{1/p}\Bigr\}
&\le
a^{-1/p}
\max\left\{
\Bigl((n-1)\bigl((a\alpha-1)^2+1\bigr)+a^2\Bigr)^{1/2},\,1
\right\}\\
&=
a^{-1/p}
\Bigl((n-1)\bigl((a\alpha-1)^2+1\bigr)+a^2\Bigr)^{1/2}.
\end{align*}
Therefore,
\[
C_{\mathrm{tr}}(\Omega_\gamma)
\le
a^{\frac1q-\frac1p}
\Bigl((n-1)\bigl((a\alpha-1)^2+1\bigr)+a^2\Bigr)^{1/2}
\,C_{\mathrm{tr}}(\Omega_n).
\]

Finally, since
\[
a=\frac{n-p}{\gamma-p},
\qquad
\alpha=\frac{\gamma-1}{n-1},
\]
we have
\[
a\alpha-1
=
\frac{(n-p)(\gamma-1)}{(\gamma-p)(n-1)}-1
=
-\frac{(p-1)(\gamma-n)}{(\gamma-p)(n-1)}.
\]
Substituting this identity and the value of \(a\) into the previous estimate,
we obtain \eqref{eq:ctr-cusp-estimate-explicit}.
\end{proof}

\vskip 0.3cm

Alexander Menovschikov; Department of Mathematics, HSE University, Moscow, Russia
 
\emph{E-mail address:} \email{menovschikovmath@gmail.com} \\

Alexander Ukhlov; Department of Mathematics, Ben-Gurion University of the Negev, P.O.Box 653, Beer Sheva, 8410501, Israel 
							
\emph{E-mail address:} \email{ukhlov@math.bgu.ac.il}

\end{document}